\documentclass[reqno,a4paper,12pt]{amsart}

\numberwithin{equation}{section}

\usepackage{amsmath,amsthm,amssymb,amsfonts,mathtools}
\usepackage{url}
\usepackage{hyperref}
\usepackage[sort,nocompress]{cite}
\mathtoolsset{showonlyrefs}

\newtheorem{definition}{Definition}[section]
\newtheorem{theorem}[definition]{Theorem}
\newtheorem{lemma}[definition]{Lemma}
\newtheorem{corollary}[definition]{Corollary}
\newtheorem{proposition}[definition]{Proposition}

\newtheorem*{theorem*}{Theorem}

\usepackage[no-math]{fontspec}
\usepackage[cachedir=minted-cache]{minted}

\newmintinline[lean]{lean4}{bgcolor=white}
\newminted[leancode]{lean4}{escapeinside=!!,
                            breaklines,
                            fontsize=\normalsize}
\newcommand{\N}{\mathbb N}
\newcommand{\Z}{\mathbb Z}
\newcommand{\R}{\mathbb R}
\newcommand{\C}{\mathbb C}
\newcommand{\Q}{\mathbb Q}
\newcommand{\HH}{\mathbb H}
\newcommand{\calT}{\mathcal T}
\newcommand{\calG}{\mathcal G}
\newcommand{\dd}{\delta}
\allowdisplaybreaks

\newcommand{\norm}[1]{\left\lVert #1 \right\rVert}
\newcommand{\QQ}{\mathcal{Q}}
\newcommand{\thetac}[2]{{\vartheta\left[\begin{smallmatrix}
        #1 \\ #2
    \end{smallmatrix}\right]}}
\usepackage{enumerate}
\newcommand{\zak}{\mathcal{Z}}

\title[Derivatives of Theta Functions and a Problem of Lyubarskii and Nes]{Derivatives of Theta Functions and a Problem of Lyubarskii and Nes}

\author[Lukas Liehr]{Lukas Liehr}
\address{Department of Mathematics, Bar-Ilan University, Ramat-Gan 5290002, Israel}
\email{lukas.liehr@biu.ac.il}

\author[Irina Shafkulovska]{Irina Shafkulovska}
\address{Faculty of Mathematics, University of Vienna, Oskar-Morgenstern-Platz 1, A-1090 Vienna, Austria}
\email{irina.shafkulovska@univie.ac.at}

\author[Mitchell A. Taylor]{Mitchell A. Taylor}
\address{Department of Mathematics, ETH Z\"urich, Ramistrasse 101, 8092 Z\"urich, Switzerland}
\email{mitchell.taylor@math.ethz.ch}

\begin{document}

\begin{abstract}
We characterize all lattices $\Lambda \subset \R^2$ of rational density for which the Gabor system generated by the first Hermite function along $\Lambda$ is a frame for $L^2(\R)$. We prove that these are precisely the lattices whose density satisfies $D(\Lambda) = \frac q p $ where $p,q \in \N$ are coprime with $q \geq p+2$, thereby confirming a conjecture of Lyubarskii and Nes. A formalization of our main result in Lean $4$ is also included.
\end{abstract}


\maketitle

\section{Introduction and result}

Motivated by a question posed by Lyubarskii and Nes \cite{lyubarskii2013gabor}, we consider the problem of determining the lattices $\Lambda$ so that the system of time-frequency shifts of the first Hermite function along $\Lambda$ forms a frame for $L^2(\R)$. Precisely, given $g\in L^2(\R)$ and a lattice $\Lambda = M\Z^2$ where $M\in\mathrm{GL}(2,\R)$, let
$$
\calG(g,\Lambda) :=\{e^{2\pi i \omega t} g(t-x) : (x,\omega)\in\Lambda\}
$$
be the system of time-frequency shifts of $g$ along $\Lambda$, which we refer to as the  \emph{Gabor system} of $g$ along $\Lambda$. We call $\calG(g,\Lambda)$ a \emph{Gabor frame} if there exist constants $0<A\leq B<\infty$ such that
\begin{equation}\label{eq:def:frame}
    A\norm{f}^2
\leq
\sum_{(x,\omega)\in\Lambda} |\langle f,M_{\omega}T_{x} g\rangle|^2 \leq
B\norm{f}^2,\quad f\in L^2(\R).
\end{equation}
The \emph{density} of the lattice $\Lambda$ is given by the reciprocal of the volume of a fundamental domain of $\Lambda$. For a separable lattice $\Lambda = \alpha\Z\times\beta\Z$ one has $D(\Lambda) = (\alpha\beta)^{-1}$. By the density theorem for Gabor frames, $\calG(g,\Lambda)$ can only be a frame if $D(\Lambda)\geq 1$ \cite{Heil2007,GroechenigKoppensteiner2019} and under mild regularity and decay assumptions, this necessary density condition becomes $D(\Lambda) > 1$ \cite{Balian1981,Low1985}.

The Hermite functions $(h_n)_{n\geq 0} \subset L^2(\R)$, defined by
\begin{equation}
    h_n(t) = e^{\pi t^2} \frac{d^n}{dt^n}e^{-2\pi t^2},
\end{equation}
form an orthogonal basis of $L^2(\R)$. In particular, $h_0(t)=e^{-\pi t^2}$ is the standard Gaussian and, up to a normalization which does not affect the frame property, the first Hermite function is given by $h_1(t)=t e^{-\pi t^2}$.

The lattices $\Lambda$ for which the Gaussian Gabor system $\calG(h_0,\Lambda)$ forms a frame were characterized by Lyubarskii \cite{Lyubarskii1992} and Seip-Wallst\'en \cite{Seip1992,SeipWallsten1992}: $\calG(h_0,\Lambda)$ is a frame if and only if $D(\Lambda)>1$. For Hermite functions of higher order, the corresponding characterization remains an open problem. In the present paper, we are interested in the Gabor frame properties of $h_1$.

Lyubarskii and Nes \cite{lyubarskii2013gabor} showed that for every odd function $g$ belonging to the Feichtinger algebra and every separable lattice $\Lambda = \alpha\Z\times\beta\Z$, the system $\calG(g,\Lambda)$ is \emph{not} a Gabor frame if the density of the lattice $\Lambda$ satisfies
\begin{equation}\label{eq:obst}
    D(\Lambda) = \frac{m+1}{m} \quad \text{for some } m\in \N.
\end{equation}
Since $h_1$ is odd, this provides an arithmetic obstruction to the frame property of the first Hermite function: within the range $D(\Lambda)>1$ permitted by the density theorem, the rational densities $\frac{m+1}{m}$, $m\in \N$, are excluded. In the same paper, Lyubarskii and Nes conjectured that for $h_1$ this is the only obstruction at rational densities in the case of separable lattices (for Hermite functions $h_n$ with $n>1$ it is known that \eqref{eq:obst} is \emph{not} the only obstruction, see the discussion in Section \ref{sec:prev_rec_work}). We confirm the conjecture for arbitrary lattices.

\begin{theorem}\label{thm:main-intro}
Let $\Lambda \subset \R^2$ be a lattice with density $D(\Lambda) \in \Q$. Then the following two statements are equivalent:
\begin{enumerate}
    \item $\calG(h_1,\Lambda)$ is a frame for $L^2(\R)$;
    \item $D(\Lambda) = \frac qp$, where $p,q \in \N$ are coprime with $q\ge p+2$.
\end{enumerate}
\end{theorem}

The substance of Theorem \ref{thm:main-intro} is the implication $(2)\Rightarrow(1)$. The converse implication follows from the original work of Lyubarskii and Nes and its extension to general lattices, see \cite{Faulhuber2020}.
We also emphasize that the relevant range is $1 < D(\Lambda) \leq 2$: by a result of Gr\"ochenig and Lyubarskii \cite{GroechenigLyubarskii2007,GroechenigLyubarskii2009}, the system $\calG(h_1,\Lambda)$ is a frame for every lattice satisfying $D(\Lambda) > 2$, while all exceptional densities $\frac{m+1}{m}$ lie in the interval $(1,2]$. Theorem \ref{thm:main-intro} settles the frame property for all rational densities in this remaining strip.

\subsection{Relation to previous and recent work}\label{sec:prev_rec_work}
Theorem \ref{thm:main-intro} should be seen in the context of the frame set problem for Hermite functions, that is, the problem of determining all lattices $\Lambda$ for which $\calG(h_n,\Lambda)$ is a frame. It was conjectured in the literature \cite[Conj.~4.2]{Groechenig2014} that if $n$ is even, then the system $\calG(h_n,\alpha\Z \times \beta \Z)$ is a frame if and only if $\alpha\beta < 1$. If $n$ is odd, the conjecture proposed that $\calG(h_n,\alpha\Z \times \beta \Z)$ is a frame if and only if $\alpha\beta < 1$ and $\alpha\beta \neq \frac{m}{m+1}$ for all $m \in \N$.

The conjecture described above was disproved for the orders $n \equiv 2,3 \pmod 4$ by Lemvig \cite{Lemvig2017}, and subsequently for all orders $n \geq 2$ by Horst, Lemvig and Videb\ae k \cite{HorstEtAl2025}. Consequently, $h_1$ is the only Hermite function of positive order for which the conjectured description of the frame set may still be correct, and Theorem \ref{thm:main-intro} confirms it for all lattices of rational density. 

Recently, in \cite{AKIZ}, Kulikov and Zlotnikov showed that for every
$
\delta\in(0,1)\setminus\left\{\frac{m}{m+1}\right\}_{m\in\mathbb N},
$
including irrational values of $\delta$, there exists $\gamma=\gamma(\delta)>0$ such that the two symmetric infinite pieces of the hyperbola
$$
\left \{(\alpha,\beta):\alpha\beta=\delta,\ \alpha>\gamma \right \}
\quad\text{and}\quad
\left \{(\alpha,\beta):\alpha\beta=\delta,\ \beta>\gamma \right \}
$$
belong to the frame set of $h_1$. However, unlike our result for rational $\delta$, which holds for all $\alpha,\beta>0$ satisfying $\alpha\beta=\delta$, their argument only applies to points for which $\alpha$ or $\beta$ is sufficiently large.

Finally, we mention the recent works \cite{FaulhuberEtAl2025,FaulhuberShafkulovskaZlotnikov_part2,Faulhuber2026wirtinger}, which further investigate the frame set problem for Hermite functions. We also refer to \cite{UlanovskiiZlotnikov2026} for recent parity-based obstructions in a more general setting. A numerical investigation of Gabor frames generated by Hermite functions is presented in \cite{GhoshSelvan2025}.

In Section \ref{sec:conclusion} we show that our methods further yield frame results for Gabor systems generated by linear combinations of Hermite functions (Corollary \ref{cor:higher_order}). This should be compared with the recent results of Ulanovskii and Zlotnikov \cite{UlanovskiiZlotnikov2026}.

\subsection*{Usage of Large Language Models}
Large language models were used as research tools throughout this project and played an important role in the development of this work. After reducing the original Gabor frame problem, via the Zibulski-Zeevi criterion, to a question concerning derivatives of theta functions, we used GPT-5.4 to explore connections between the resulting problem and the existing literature on theta functions and abelian varieties. In particular, this exploration led us to the elliptic determinant identities of Rosengren and Schlosser \cite{RosengrenSchlosser2006} and to the jet-independence results of Bauer and Szemberg \cite{BauerSzemberg1997}, which became key ingredients in the proof of our main result. GPT-5.4 was also used to help navigate the different terminology appearing in the literature on theta functions. While our problem is formulated in terms of classical theta functions and their derivatives, parts of the literature needed for the proof are formulated in the language of line bundles, their sections, and jet separation on abelian varieties. In these cases, GPT-5.4 was used to clarify the correspondence between these formulations. The Lean formalization of the main result was done with the help of Codex-5.4 and Opus 4.6.

\section{Preliminaries}

\subsection{Gabor frames}

We begin with some standard facts from Gabor theory; see \cite{Groechenig2001,Folland1989} for a detailed exposition and \cite{GroechenigKoppensteiner2019} for an updated overview.

It is well-known that $\calG(g,\Lambda)$ is a frame with frame bounds $0<A\leq B<\infty$ if and only if its \emph{frame operator}
\begin{equation}
    S_{g,\Lambda} f=\sum_{(x,\omega)\in\Lambda} 
    \langle f,M_{\omega}T_{x} g\rangle 
    M_{\omega}T_{x} g, \quad f\in L^2(\R),
\end{equation}
is a positive definite operator with spectrum in $[A,B]$. If this is the case, then any $f\in L^2(\R)$ can be recovered via
\begin{equation}
    f = \sum_{(x,\omega)\in\Lambda} 
    \langle f,M_{\omega}T_{x} g\rangle 
    S_{g,\Lambda}^{-1} M_{\omega}T_{x} g,
\end{equation}
where the series on the right converges unconditionally in $L^2(\R)$.
It is known that $D(\Lambda)\geq 1$ is a necessary condition so that $\calG(g,\Lambda)$ is a frame \cite{Heil2007}.
Assuming a rational relation $D(\Lambda)\in\Q$, this additional structure enables us to study the frame property of $\calG(g,\Lambda)$ through spectral properties of the Zak transform of $g$. For definitions and facts about the Zak transform, we refer to Section \ref{sec:zak}.

\subsection{Symplectic matrices and metaplectic operators}

Every lattice $A\Z^2$ with $A\in\mathrm{GL}(2,\R)$ can also be written as $A'\Z^2$ with $\det A'>0$. We may therefore always assume that $\det A>0$. It follows from the Iwasawa decomposition \cite[Chap.~VI.~4.]{Knapp1996} of the special linear group $\mathrm{SL}(2,\R)$ that any $A\in\mathrm{GL}(2,\R)$ with $\det A>0$ can be decomposed as
\begin{equation}\label{eq:A_decomposed}
A  = (\det A)^{1/2} \cdot \begin{psmallmatrix}
      1 & 0 \\ \eta & 1
  \end{psmallmatrix} \cdot \begin{psmallmatrix}
      \nu & 0 \\ 0 & \nu^{-1}
  \end{psmallmatrix}  \cdot \begin{psmallmatrix}
      \mathrm{Re}(\sigma) & \mathrm{Im}(\sigma) \\ -\mathrm{Im}(\sigma) & \mathrm{Re}(\sigma) 
       \end{psmallmatrix} 
       \eqqcolon (\det A)^{1/2} \cdot V_\eta \cdot D_\nu \cdot R_{\sigma},
\end{equation}
where $\eta\in\R, \, \nu>0$ and $|\sigma|=1$.
The double cover of $\mathrm{SL}(2,\R)$ is the metaplectic group $\mathrm{Mp}(2,\R)$, and it is usually realized as a group of unitary operators acting on $L^2(\R)$. Since $V_\eta, D_\nu, R_\sigma$ generate the group, all metaplectic operators are generated from the following three types:
\begin{equation}
    \widehat{V}_\eta f(t) = e^{\pi i \eta t^2}f(t),\quad \widehat{D}_\nu f(t) = \nu^{-1/2} f(\tfrac{t}{\nu}),\quad \widehat{R}_\sigma h_n(t) = (\overline{\sigma})^n h_n, \quad n\in\N_0.
\end{equation}
As $(h_n)_{n\in\N_0}$ is an orthogonal basis of $L^2(\R)$, this fully describes the action of $\widehat{R}_{\sigma}$. 
Note that $\widehat{R}_{1}=\mathrm{id}$ and $\widehat{R}_{i}=\mathcal{F}$ is the Fourier transform.
The projection from $\mathrm{Mp}(2,\R)$ to $\mathrm{SL}(2,\R)$
acts as
\begin{equation}\label{eq:pi_Mp}
\widehat{V}_\eta\widehat{D}_\nu\widehat{R}_\sigma \mapsto V_\eta D_\nu R_\sigma.
\end{equation}
Matrices $S\in \mathrm{SL}(2,\R)$ and metaplectic operators $\widehat{S}\in\mathrm{Mp}(2,\R)$ interact particularly well with time and frequency shifts. 
In particular, they intertwine Gabor frames in the following sense \cite[p.~200]{Groechenig2001}:
\begin{equation}\label{eq:Gabor_Mp_invariance}
    \calG(g,\Lambda) \text{ is a Gabor frame}\quad \iff \quad \calG(\widehat{S}g,S\Lambda) \text{ is a Gabor frame.}
\end{equation}
We refer to \cite{Groechenig2001,Folland1989} for a detailed exposition and the connection between Gabor frames and representation theory on the Heisenberg group. 

\subsection{The Zak transform}\label{sec:zak}

The Zak transform of a function $f$ is formally defined by 
\begin{equation}
    \zak f(x,\omega)
\coloneqq
\sum_{k\in\Z}f(x- k)e^{2\pi i k\omega}.
\end{equation}
The formula converges pointwise everywhere to a continuous function if $f$ belongs to the Wiener amalgam space $W(C,\ell^1)$, which is the collection of all continuous $f$ such that
$$
\|{f}\|_{W(L^\infty, \ell^1)} =
\sum\limits_{k\in\Z} \sup_{x\in[0,1]} |f(x-k)| <\infty.
$$
With the usual density arguments, $\zak$ extends to a 
unitary mapping from $L^2(\R)$ to $L^2([0,1]^2)$.

The Zak transform $\mathcal{Z} f$ is a quasi-periodic function in the sense that 
\begin{equation}\label{eq:Z_quasip}
  \zak f(x+1,\omega)=e^{2\pi i\omega}\zak f(x,\omega),
\quad
\zak f(x,\omega+1)=\zak f(x,\omega).
\end{equation}
Assuming that the Fourier transform $\hat{f}(\xi) = \int_\R f(t) e^{-2\pi i t \xi} \, dt$ is well-defined,
\begin{equation}\label{eq:Zak_ft}
  \zak f(x,\omega) = e^{2\pi i x\omega} \zak \hat{f}(\omega, -x).  
\end{equation}
The key observation connecting the Zak transform to Gabor frames is the fact that the Zak transform is an intertwining operator between the frame operator and a multiplication operator \cite[Thm.~8.3.1.]{Groechenig2001}. Namely, we have
\begin{equation}\label{eq:frame_vs_Zak}
    \zak (S_{g,\Z^2} f) = |\zak g|^2\cdot \zak f.
\end{equation}
Invertibility of multiplication operators is significantly simpler to check than the invertibility of a general frame operator.

Zibulski and Zeevi extended \eqref{eq:frame_vs_Zak} from $\Z^2$ to separable lattices $\alpha\Z\times\beta\Z$ with $\alpha\beta\in\Q$ by developing a multi-window approach and a vector-valued Zak transform \cite{ZibulskiZeevi1993, ZibulskiZeevi1997}. There are several equivalent formulations. In the following, we refer to the version in \cite[Thm.~7.2.]{GroechenigKoppensteiner2019}.

\begin{theorem}\label{thm:ZZ_crit_cited}
    Let $g\in L^2(\R)$ and let $p,q\in \N$ be coprime with $p<q$. 
    Then $\calG(g,\tfrac{p}{q}\Z\times \Z)$ is a Gabor frame if and only if the singular values of the matrix $\QQ(x,\omega)\in \C^{q\times p}$ with entries 
    \begin{equation}
        \QQ(x,\omega)_{rs} = \zak g( x+\tfrac{pr}{q},  \omega +\tfrac{s}{p}) ,\quad\, 0\leq r\leq q-1, 0\leq s\leq p-1,
    \end{equation}
    are contained in a compact interval in $(0,\infty)$ for almost all $x,\omega\in \R$.
\end{theorem}

We conclude this section by recalling that the Zak transform of $\widehat{V}_\eta \widehat{D}_\nu\widehat{R}_{\sigma} h_1$ is given by
\begin{equation}\label{eq:Zak_general_h1}
\zak (\widehat{V}_\eta \widehat{D}_\nu\widehat{R}_{\sigma} h_1) (x,\omega) = -\tfrac{\overline{\sigma}}{\nu^{3/2}}
        \sum_{k\in\Z}  (k-x) e^{\pi i (\eta+i/\nu^2)(k-x)^2} e^{2\pi i k\omega}.
\end{equation}

\subsection{Theta functions}\label{sec:theta}

The theory of theta functions starts with the series
\begin{equation}\label{eq:def:theta}
    \vartheta(z,\tau) = \sum\limits_{k\in\Z}e^{\pi i \tau k^2}e^{2\pi i z k},\quad z\in \C,\, \tau\in \C \text{ with } \Im(\tau)>0.
\end{equation}
The above function is an entire quasi-periodic function in the first variable $z$ and has natural connections to elliptic functions and their applications, see \cite{MumfordTataI,FarkasKra2001}. The behavior in the second variable is more subtle, and it is intrinsically tied to the theory of modular forms on $\HH = \{\tau \in \C: \Im(\tau)>0\}$ \cite{MumfordTataI}.
We refer to \cite{MumfordTataI, Krazer1903,FarkasKra2001} for a comprehensive introduction to the theory of theta functions; we mostly adopt the notation used in \cite{MumfordTataI}.

The function $\vartheta$ is commonly referred to as Jacobi's third theta function. It is a quasi-periodic function satisfying 
\begin{equation}
\vartheta(z+1,\tau)=\vartheta(z,\tau),
\quad
\vartheta(z+\tau,\tau)=e^{-\pi i\tau-2\pi i z}\vartheta(z,\tau).
\end{equation}
Up to scaling, $\vartheta(\cdot,\tau)$ is the unique entire function satisfying this quasi-periodicity condition \cite[Chap.~2,~Cor.~1.9.]{FarkasKra2001}. 
More generally, we consider the theta functions with characteristic $\begin{psmallmatrix}
    a\\b
\end{psmallmatrix}\in\C^2$, defined as
\begin{equation}\label{eq:def:thetac}
    \vartheta\left[\begin{smallmatrix}
        a \\ b
    \end{smallmatrix}\right](z,\tau) =\sum\limits_{k\in\Z}e^{\pi i \tau (k+a)^2}e^{2\pi i (z+b) (k+a)},\quad z\in \C,\, \tau\in \HH.
\end{equation}
The most important examples come from characteristics in $\Q^2$, but we will consider complex-valued characteristics here. 
Note that, up to Gaussian factors, $\thetac{a}{b}$ is just a shift of $\vartheta$ \cite[p.~73]{FarkasKra2001},
\begin{equation}\label{eq:thetac_vs_vartheta}
    \thetac{a}{b}(z,\tau) = e^{\pi i a^2\tau + 2\pi i (az +ab)}\vartheta(z+b+a\tau, \tau),
\end{equation}
so $\thetac{a}{b}$ can be seen merely as a translation of $\vartheta$.
Just as $\vartheta$, the function $\thetac{a}{b}$ is an entire function uniquely characterized by its quasi-periodicity
\begin{equation}\label{eq:thetac_quasip}
\thetac{a}{b}(z+1,\tau)=e^{2\pi i a}\thetac{a}{b}(z,\tau),
\ \
\thetac{a}{b}(z+\tau,\tau)=e^{-\pi i(2b+2z +\tau)}\thetac{a}{b}(z,\tau).
\end{equation}
The  function $\thetac{a}{b}$ only has simple zeros, given by 
\begin{equation}\label{eq:zeros}
    \thetac{a}{b}(z,\tau)=0 \quad\Leftrightarrow\quad z \in \tfrac{1-2b}{2} +\tfrac{1-2a}{2}\tau +\Z+\tau\Z.
\end{equation}
In other words, $\thetac{a}{b}$ has a unique simple zero in the 
fundamental period parallelogram $[0,1)+\tau[0,1)$. 

We are interested in a slightly different version of quasi-periodic entire functions.
Let $\mathcal O(\C)$ denote the space of entire functions on $\C$ and let $p\in\N$. The space of $p$-th order $\vartheta$-functions with characteristic $\begin{psmallmatrix}
    a\\b
\end{psmallmatrix}\in\C^2$
is given by 
\begin{equation}\label{eq:def:Tp}
    \calT_{p,a,b}(\tau):=
\left\{f\in\mathcal O(\C):
 f(z+1)=e^{2\pi i a}f(z),\ 
 f(z+\tau)=e^{-\pi i p\tau-2\pi i (pz+b)}f(z)
\right\}.
\end{equation}
This space appears, for instance, in \cite[Chap.~2,~Def.~7.4.]{FarkasKra2001}. If $a=b=0$, we simply write $\calT_p(\tau) := \calT_{p,0,0}$.
We make use of the following proposition \cite[Chap.~2, p.~132--134]{FarkasKra2001}.

\begin{proposition}\label{prop:elem_properties_calTp}
Let $\tau\in\HH$,  $p\in\N$, and  $a,b\in\C$. The following statements hold:
\begin{enumerate}[(i)]
    \item 
    The space $\calT_{p,a,b}(\tau)$ is a $p$-dimensional $\C$-vector space.
    \item The entire functions
    \begin{equation}
    \begin{split}
    z\mapsto
        \thetac{(a+s)/p}{b}(pz,p\tau),\quad 0\leq s\leq p-1,
    \end{split}
    \end{equation}
    form a canonical basis of $\calT_{p,a,b}(\tau)$. 
    \item Any nonzero $f\in\calT_{p,a,b}(\tau)$ has exactly $p$ zeros
$z_1,\dots,z_p$, counted with multiplicity, in the half-open fundamental
period parallelogram
\[
\mathcal P_\tau:=[0,1)+\tau[0,1)
=\{x+y\tau:0\le x<1,\ 0\le y<1\},
\]
and their sum satisfies
\begin{equation}\label{eq:p_sum_zeros}
    \sum_{l=1}^p z_l
    \in \tfrac{p-2b}{2}+\tfrac{p-2a}{2}\tau+\Z+\tau\Z.
\end{equation}
Conversely, any multiset $z_1,\dots,z_p\in\mathcal P_\tau$ satisfying
\eqref{eq:p_sum_zeros} is the zero multiset of a nonzero function in
$\calT_{p,a,b}(\tau)$.
\end{enumerate}  
\end{proposition}

\section{Derivatives of theta-functions}\label{sec:thetaD}

We now shift our focus to derivatives of theta functions. To do so, we define the differential operators 
\begin{equation}
    \dd_z\coloneqq\frac{1}{2\pi i}\frac{d}{dz}\quad \text{and}
    \quad 
    P(\dd_z) =  \sum\limits_{d=0}^{m}c_d \dd_z^d, \quad \text{where}\quad P(X) = \sum\limits_{d=0}^{m}c_d X^d.
\end{equation}
Due to the Gaussian decay of the series expansion in \eqref{eq:def:thetac}, it follows that
\begin{equation}\label{eq:delta_theta_ab}
   \delta_z^m\thetac{a}{b}(z,\tau) = 
   \sum\limits_{k\in\Z}(k+a)^me^{\pi i \tau (k+a)^2}e^{2\pi i (z+b) (k+a)},\quad z\in\C,\, \tau\in\HH,\quad m\in\N_0.
\end{equation}

\begin{lemma}\label{lem:torsion-jet}
Let $\tau\in\HH$, $q\in\N, \, q\geq 2$ and let 
$\Theta(z) = \thetac{a}{b}(z,\tau)$ for some $a,b\in\C$.
Let $P$ be a
nonzero polynomial with
$$
\deg P\leq q-2.
$$
Then, for every $w\in\C$, the system of equations
\begin{equation}\label{eq:impossible_system}
    P(\dd_z)\Theta\left(w+\frac{j}{q}\right)=0,
\quad j=0,\dots,q-1,
\end{equation}
cannot be satisfied.
\end{lemma}
\begin{proof} 
Suppose that \eqref{eq:impossible_system} holds for some $P\neq 0$ of degree $m\leq q-2$, i.e., 
$$
P(X)=\sum\limits_{d=0}^{m}c_d X^d
$$
for some non-zero vector $c\in\C^{m+1}$.
By \cite[Lem.~1.4.]{FarkasKra2001} and \eqref{eq:thetac_quasip} we have
\begin{equation}\label{eq:jet_changebase}
    \thetac{a}{b}(z+\tfrac{j}{q},\tau)= \sum_{s=0}^{q-1} \thetac{(a+s)/q}{qb}(qz+j,q^2\tau) = \sum_{s=0}^{q-1} e^{2\pi i (a+s)j/q}\thetac{(a+s)/q}{qb}(qz,q^2\tau).
\end{equation}
By Proposition \ref{prop:elem_properties_calTp} (ii), the functions $\thetac{(a+s)/q}{qb}(qz,q^2\tau)$, $0\leq s\leq q-1$, yield a basis of $\calT_{q,a,qb}(q\tau)$.
The relations \eqref{eq:jet_changebase} describe a change of basis because 
 $\left(e^{2\pi i (a+s)j/q}\right)_{0\leq j,s\leq q-1}$ is the Fourier matrix multiplied by a diagonal invertible matrix.
 Hence, $\phi_s(z)=\thetac{a}{b}(z+\tfrac{s}{q},\tau)$,  $0\leq s\leq q-1$, is also a basis of  $\calT_{q,a,qb}(q\tau)$.

We can rewrite \eqref{eq:impossible_system} in matrix form as 
\begin{equation} \left(\dd_z^d\phi _s(w)\right)_{\substack{0\leq s\leq q-1 \\ 0\leq d\leq m}} \cdot c = 0\in \C^{q}.
\end{equation}
By \cite[Thm.~1]{BauerSzemberg1997},
$\left(\dd_z^d \phi_s(w)\right)_{\substack{0\leq s\leq q-1 \\ 0\leq d\leq m}}$ has full rank. Hence, $c=0$, contradicting $P\neq 0$. 

\end{proof}

\begin{lemma}
\label{lem:frobenius}
Let $p\in\N$, $b\in\C$ and $\tau\in\HH$.
Let $\phi_0,\dots,\phi_{p-1}$ be a basis of $\calT_{p,0,b}(\tau)$. Define
$$
\Delta(z_0,\dots,z_{p-1})
:=
\det[\phi_s(z_l)]_{l,s=0}^{p-1},
\quad
S(z):=\sum_{l=0}^{p-1}z_{l}.
$$
Then there exists a nonzero constant $C$ such that
$$
\Delta(z_0,\dots,z_{p-1})
=
C\, \thetac{1/2+p/2}{1/2+b+p/2}(S(z), \tau) \prod\limits_{0\leq l <l'< p}\thetac{1/2}{1/2}(z_{l'}-z_l,\tau).
$$
\end{lemma}
\begin{proof} 
    By the determinant formula \cite[Prop.~3.4.]{RosengrenSchlosser2006} there exists a constant $C\in\C$ such that
    \begin{equation*}
       \begin{split}
           \det[\phi_s(z_l)]_{l,s=0}^{p-1} & = C_0 
           e^{\pi i p S(z)}\, \thetac{1/2}{1/2}(S(z) +b+\tfrac{p}{2}+\tfrac{p}{2}\tau, \tau) 
    \hspace{-3pt}\prod\limits_{0\leq l <l'< p}
           \hspace{-3pt}\thetac{1/2}{1/2}(z_{l'}-z_l,\tau).
       \end{split} 
    \end{equation*}
    By \eqref{eq:thetac_vs_vartheta}, we have
    \begin{equation}
 \begin{split}
        &\hspace{16pt}e^{2\pi i p S(z)/2}\, \thetac{1/2}{1/2}(S(z) +b+\tfrac{p}{2}+\tfrac{p}{2}\tau, \tau) \\
        & =
        e^{ -\frac{\pi i p^2 \tau}{4} - \pi i p b - \frac{\pi i p (p+1)}{2} }\, \thetac{1/2+p/2}{1/2+b+p/2}(S(z), \tau).
 \end{split}
    \end{equation}
    The exponential term on the right-hand side is independent of $z$, so we can absorb it together with $C_0$ into a new constant $ C$ and obtain
    \begin{equation}
        \det[\phi_s(z_l)]_{l,s=0}^{p-1} = C\, \thetac{1/2+p/2}{1/2+b+p/2}(S(z), \tau) \prod\limits_{0\leq l <l'<p}\thetac{1/2}{1/2}(z_{l'}-z_l,\tau).
    \end{equation}
Since $\phi_0,\dots ,\phi_{p-1}$ are linearly independent, their evaluation determinant is not identically zero. Hence, the constant $C$ is not zero.
\end{proof}

\begin{theorem}\label{thm:primitive-derivative}
Let $a,b\in\C$ and $n,p,q\in\N$ satisfy
$$
(p,q)=1,
\quad
q\geq {n}p+2.
$$
For every $w\in\C$, every $\tau \in \mathbb H$ and every polynomial $P$ of degree $n$, the map
$$
\calT_{p,a,b}(\tau)\to \C^q,
\quad
f \mapsto
\left({P(\dd_z)} f\left(w+\tfrac{r}{q}\right)\right)_{r=0}^{q-1}
$$
is injective.
\end{theorem}

\begin{proof}
{
Since 
\begin{equation}
    \calT_{p,a,b}\to \calT_{p,0,b+a\tau}, \quad f\mapsto e^{-2\pi i a\cdot} f
\end{equation}
is an isomorphism and 
\begin{equation}
    \dd_z f(z) = \dd_z \left(e^{2\pi i az}g(z)\right)= e^{2\pi i a z}(a g(z) + \dd_z g(z)),\quad g(z) = e^{-2\pi i az} f(z),
\end{equation}
by Leibniz' rule there exists a polynomial $Q$ of degree at most $n$ such that
\begin{equation}\label{eq:P_vs_Q_for_a}
    P(\dd_z) f (z) = e^{2\pi i a z} Q(\dd_z) g(z),\quad g(z) = e^{-2\pi i az} f(z). 
\end{equation}
The polynomial $Q$ has degree $n$ because its leading coefficient equals the leading coefficient of $P$. Thus, we can assume without loss of generality that $a=0$ and that $P$ has a leading coefficient equal to one.
}

{Let $\phi_0,\dots,\phi_{p-1}$ be a basis of $\calT_{p,0,b}(\tau)$.}  Suppose, to
the contrary, that the stated map is not injective.  Then the
$q\times p$ matrix
$$
R_{r,s}:={{P(\dd_z)}}\phi_s\left(w+\tfrac{r}{q}\right),
\quad
r=0,\dots,q-1,
\quad
s=0,\dots,p-1,
$$
has rank strictly smaller than $p$.  Hence, all of its $p\times p$ minors
vanish.

For $j\in\Z$, define the consecutive minor
$$
M_j:=
\det\left[
{P(\dd_z)}\phi_s\left(w+\tfrac{j+l}{q}\right)
\right]_{l,s=0}^{p-1}.
$$
Each $P(\dd_z)\phi_s$ is $1$-periodic, so reducing the indices
$j+l$ modulo $q$ only permutes the rows. 
Rank failure gives
$$
M_j=0,
\quad j=0,\dots,q-1.
$$
Let
$$
\Delta(z_0,\dots,z_{p-1})
:=
\det[\phi_s(z_l)]_{l,s=0}^{p-1}.
$$
By Lemma~\ref{lem:frobenius}, we have
$$
\Delta(z_0,\dots,z_{p-1})
=
C\, \thetac{1/2+p/2}{1/2+b+p/2}(S(z), \tau)V(z_0,\dots,z_{p-1}) ,
$$
where
$$
S(z_0,\dots,z_{p-1}):=\sum_{l=0}^{p-1} z_l,
\quad
V(z_0,\dots,z_{p-1})
:=\prod\limits_{0\leq l<l'< p}\thetac{1/2}{1/2}(z_{l'}-z_l,\tau).
$$
We apply $P(\dd_z)$ to each row of $[\phi_s(z_l)]_{l,s=0}^{p-1}$ with respect to $z_l$ and denote this by 
$$
P(\dd_{z_0})\cdots P(\dd_{z_{p-1}}) [\phi_s(z_l)]_{l,s=0}^{p-1}.
$$
We then restrict the argument to 
$$
z_l=w+\tfrac{l}{q},
\quad l=0,\dots,p-1.
$$
On the left, multilinearity of the determinant gives
$$
\left.
P(\dd_{z_0})\cdots P(\dd_{z_{p-1}})\Delta
\right|_{z_l=w+l/q}
=
\det\left[
P(\dd_z)\phi_s\left(w+\frac{l}{q}\right)
\right]_{l,s=0}^{p-1}
=:M(w).
$$
At the restricted points, we have
$$
z_{l'}-z_l=\tfrac{l'-l}{q}.
$$
As $1\leq l'-l\leq p-1<q$, none of these differences lies in
$\Z+\tau\Z$.  Since $\thetac{1/2}{1/2}(\cdot,\tau)$ has its only zero at the
origin modulo $\Z+\tau\Z$, it follows that
$$
V_0:=
V\left(w,w+\tfrac1q,\dots,w+\tfrac{p-1}{q}\right)
=
\prod\limits_{0\leq l <l'< p}\thetac{1/2}{1/2}(\tfrac{l'-l}{q},\tau)
\neq0.
$$
The above number is independent of $w$.
Let 
\begin{equation}
    \Theta(z)=\thetac{1/2+p/2}{1/2+b+p/2}(z,\tau).
\end{equation}
By Leibniz' rule, after the restriction $z_l=w+l/q$, every derivative of
$V$ is evaluated at the fixed tuple of differences
$\{(l'-l)/q\}_{l<l'}$ and is therefore a constant.  Hence, there are
constants $c_0,\dots,c_{np}$ such that
$$
M(w)
=
\sum_{d=0}^{np}c_d
\dd_z^d\Theta \left(pw+\tfrac{p(p-1)}{2q}\right).
$$
Equivalently,
$$
M(w)=Q(\dd_z)\Theta \left(pw+\tfrac{p(p-1)}{2q}\right),
$$
for a polynomial $Q$ of degree at most $np$.

The coefficient $c_{np}$ of the derivative of the highest order  comes only from the term in which the highest order
derivative part of every operator $P(\dd_{z_l})$
falls on the factor $\Theta(S(z_0,\dots,z_{p-1}))$ and no derivative falls on $V$. Therefore,
$$
c_{np}
=
C\prod_{0\leq l<l'< p}
\thetac{1/2}{1/2}\left(\tfrac{l'-l}{q},\tau\right)
=
CV_0
\neq0.
$$
Thus, $Q\neq0$ and $\deg Q={{n}}p$.

Since $M_j=M(w+j/q)$, the vanishing of the consecutive minors gives
$$
Q(\dd_z)\Theta\left(
pw+\tfrac{pj}{q}+\tfrac{p(p-1)}{2q}
\right)=0,
\quad j=0,\dots,q-1.
$$
Since $(p,q)=1$, for each $j$ there exists a residue $j'\in\{0,\dots,q-1\}$
and an integer $n_j$ such that
$$
\frac{pj}{q}=\frac{j'}{q}+n_j.
$$
Let
$$
w':=pw+\frac{p(p-1)}{2q}.
$$
The preceding vanishing can be expressed as
$$
Q(\dd_z) \Theta\left(w'+\tfrac{j'}{q}+n_j\right)=0,
\quad j'=0,\dots,q-1,
$$
after reordering the residues. 
By \eqref{eq:thetac_quasip}, $\Theta(z+n_j)=e^{2\pi i(1/2+p/2)n_j}\Theta(z)$ for all $z\in\C$, and hence $Q(\dd_z)\Theta(z+n_j)=e^{2\pi i(1/2+p/2)n_j}Q(\dd_z)\Theta(z)$. Since this factor is nonzero, it follows that
$$
Q(\dd_z)\Theta\left(w'+\tfrac{j}{q}\right)=0,
\quad j=0,\dots,q-1.
$$
Finally,
$$
\deg Q={{n}}p\leq q-2.
$$
Since $\Theta=\thetac{1/2+p/2}{1/2+b+p/2}(\cdot,\tau)$ and $q\geq np+2\geq 2$, this contradicts Lemma~\ref{lem:torsion-jet}. Hence, the original derivative evaluation map is injective.
\end{proof}

\section{Proof of the main result}

\begin{proof}[Proof of Theorem \ref{thm:main-intro}]

It suffices to show that condition (2) implies condition (1). To this end, let
$\Lambda=A\Z^2$ be a lattice with $A\in\mathrm{GL}(2,\R)$ and density
$D(\Lambda)=\tfrac{q}{p}$, where $q,p \in \N$ are coprime and $q\geq p+2$.

\textbf{Step 1.} 
As explained before in \eqref{eq:A_decomposed}, we may assume that $\det A>0$;
we then have $\det A=p/q$.  Set
$$
B=\begin{psmallmatrix}p/q&0\\0&1\end{psmallmatrix},
\quad S=BA^{-1}\in\mathrm{SL}(2,\R).
$$
Note that $SA\Z^2=B\Z^2=\tfrac pq\Z\times\Z$.  Write the Iwasawa
decomposition of $S$ as $S=V_{\eta_0}D_{\nu_0}R_{\sigma_0}$.  Since
$\widehat R_{\sigma_0}h_1=\overline{\sigma_0}h_1$, the unimodular scalar
$\overline{\sigma_0}$ does not affect the frame property.  Hence, by
\eqref{eq:Gabor_Mp_invariance}, $\calG(h_1,A\Z^2)$ is a frame if and only if
$$
\calG(g_0,\tfrac pq\Z\times\Z)\hspace{2mm}\text{is a frame}, \ \text{where}
\hspace{2mm} g_0:=\widehat V_{\eta_0}\widehat D_{\nu_0}h_1.
$$
We are now in the setting of Theorem \ref{thm:ZZ_crit_cited}.

     \textbf{Step 2.}
     Let 
     \begin{equation}
         \QQ(x,\omega)_{rs} = \zak g_0( x+\tfrac{pr}{q}, \omega +\tfrac{s}{p}) ,\quad 0\leq r\leq q-1,\, 0\leq s\leq p-1.
     \end{equation}
     We already computed $\zak (\widehat{V}_{\eta_0} \widehat{D}_{\nu_0}h_1)$ in \eqref{eq:Zak_general_h1}, but to make the criterion in Theorem \ref{thm:ZZ_crit_cited}
     compatible with the Theorems in Sections \ref{sec:theta} and \ref{sec:thetaD}, we combine it with \eqref{eq:Zak_ft}.

     Since the Fourier transform is the metaplectic operator $\widehat{R}_i$, the composition
     $\widehat{R}_i \widehat{V}_{\eta_0} \widehat{D}_{\nu_0}$ is projected to the matrix $R_i V_{\eta_0} D_{\nu_0}\in \mathrm{SL}(2,\R)$. By the Iwasawa decomposition \eqref{eq:A_decomposed}, there exist $\eta\in\R$, $\nu>0$ and $|\sigma|=1$ such that
     \begin{equation}
         R_i V_{\eta_0} D_{\nu_0} = V_{\eta} D_{\nu} R_\sigma.
     \end{equation}
     Choose compatible metaplectic lifts of this identity; the possible overall sign
     is immaterial. Therefore, by \eqref{eq:Zak_ft} and \eqref{eq:pi_Mp} we have
     \begin{equation}
         \begin{split}
             \zak (\widehat{V}_{\eta_0} \widehat{D}_{\nu_0}h_1) (x,\omega) &= 
         e^{2\pi i x\omega} \zak (\widehat{R}_i \widehat{V}_{\eta_0} \widehat{D}_{\nu_0} h_1)(\omega, -x)  \\
         &= e^{2\pi i x\omega} \zak (\widehat{V}_{\eta} \widehat{D}_{\nu}\widehat{R}_\sigma h_1)(\omega, -x).
         \end{split}
     \end{equation}
     By \eqref{eq:Zak_general_h1} and \eqref{eq:delta_theta_ab} we may write
     \begin{equation}
         \begin{split}
                          \zak (\widehat{V}_{\eta_0} \widehat{D}_{\nu_0}h_1) (x,\omega)
        &= -\tfrac{\overline{\sigma}}{\nu^{3/2}}
        \sum_{k\in\Z}(k-\omega) e^{\pi i (\eta+i/\nu^2)(k-\omega)^2} e^{-2\pi i (k-\omega) x} \\
             &= \tfrac{\overline{\sigma}}{\nu^{3/2}}
        \sum_{k\in\Z}(k+\omega) e^{\pi i (\eta+i/\nu^2)(k+\omega)^2} e^{2\pi i (k+\omega) x} \\
        & = \tfrac{\overline{\sigma}}{\nu^{3/2}} \delta_z\thetac{\omega}{0}(x,\tau),\quad \tau = \eta+i/\nu^2
        .
         \end{split}
     \end{equation}
Thus,
\begin{equation}
        \QQ(x,\omega)_{rs}
    = 
    \tfrac{\overline{\sigma}}{\nu^{3/2}} \delta_z\thetac{\tfrac{p\omega+s}{p}}{0}(p(\tfrac{x}{p}+\tfrac{r}{q}),\tau).
\end{equation}

\textbf{Step 3.}
We apply 
Theorem~\ref{thm:primitive-derivative} for the space $\calT_{p,p\omega,0}(\tau/p)$ with $n=1$ and $P(X)=X$. By Proposition \ref{prop:elem_properties_calTp}~(ii), the functions
\begin{equation}
\phi_s(z):=\thetac{(p\omega+s)/p}{0}(pz,\tau),\quad 0\leq s\leq p-1,
\end{equation}
form a basis of this space.  
The chain rule gives
\begin{equation}
\dd_z \phi_s(w)
=p\left.\dd_z\thetac{(p\omega+s)/p}{0}(z,\tau)\right|_{z=pw}.
\end{equation}
Consequently, up to the common nonzero scalar
$\overline\sigma/(p\nu^{3/2})$, 
the matrix $\QQ(x,\omega)$ is the matrix
representation of the linear map
\begin{equation}
    \calT_{p,p\omega,0}(\tfrac{\tau}{p})\to \C^{q},\quad f\mapsto
\left(\dd_z f\left(\tfrac{x}{p}+\tfrac{r}{q}\right)\right)_{r=0}^{q-1}.
\end{equation}
Since $(p,q)=1$ and
$q\geq p+2$, the matrix $\QQ(x,\omega)$ has full rank for every $x$ and every
$\omega$.  
By the quasi-periodicity of the Zak transform \eqref{eq:Z_quasip}, the matrices
$\QQ(x+n,\omega+m)$ and $\QQ(x,\omega)$ differ by left and right
multiplication by unitary diagonal and permutation matrices. Hence, they have
the same singular values for all $x,\omega\in\R$ and $n,m\in\Z$.
On the compact set $[0,1]^2$, the singular values of $\QQ(x,\omega)$ depend continuously on its entries. 
Since $g_0$ is a Schwartz function, it belongs to $W(C,\ell^1)$ and the entries of $\QQ(x,\omega)$ depend continuously on $(x,\omega)$.
Therefore, compactness arguments guarantee that the full rank implies that all singular values of $\QQ(x,\omega)$ lie in a compact interval $[A,B]\subseteq(0,\infty)$ independent of $(x,\omega)$.
The rational Zak criterion in
Theorem~\ref{thm:ZZ_crit_cited} then implies that
$\calG(g_0,\tfrac{p}{q}\Z\times \Z)$ is a frame. By Step~1,
$\calG(h_1,A \Z^2)$ is therefore a frame.
\end{proof}

\section{Concluding remarks}\label{sec:conclusion}

We note that Theorem \ref{thm:main-intro} has a straightforward generalization to linear combinations of higher-order Hermite functions.

\begin{corollary}\label{cor:higher_order}
    Let $\Lambda=A\Z^2$ be a lattice such that $A\in\mathrm{GL}(2,\R)$ with $D(\Lambda)=\frac qp$, where $n, p,q\in\N$ and $p,q$ are coprime.
If
$
q\geq n p+2
$ 
and $g\in L^2(\R)$ is a non-trivial linear combination of $h_0,\dots, h_{n}$, then 
$
\calG(g,\Lambda)
$
is a frame for $L^2(\R)$.
\end{corollary}
\begin{proof}
\textbf{Step 1.}
Let $g\in L^2(\R)$ be a non-trivial linear combination of Hermite functions, i.e.,
\begin{equation}
    g(t)=\sum\limits_{j=0}^n c_jh_j(t),\quad c_j\in\C.
\end{equation}
If $c_j=0$ for every $1\leq j\leq n$, then $g=c_0h_0$ with
$c_0\neq0$. Since $n\in\N$ and $q\geq np+2>p$, we have
$D(\Lambda)=\tfrac{q}{p}>1$. The characterization of Gaussian Gabor frames
\cite{Lyubarskii1992,Seip1992,SeipWallsten1992} therefore proves the
claim in this case.

We may thus assume that $c_j\neq0$ for some $1\leq j\leq n$.
Replacing $n$, if necessary, by the largest such index, we may assume
that $c_n\neq0$. The hypothesis $q\geq np+2$ remains valid after this
replacement. For all $\sigma\in\C$, $|\sigma|=1$, we may write
\begin{equation}\label{eq:lin_comb_sigma_action}
    \widehat{R}_\sigma g = \sum\limits_{j=0}^n c_j \widehat{R}_\sigma h_j = \sum\limits_{j=0}^n c_j (\overline{\sigma})^j h_j
\end{equation}
and $c_n (\overline{\sigma})^n\neq0$, so  $\widehat{R}_\sigma g$ is still a non-trivial linear combination of Hermite functions. 
Apply the lattice reduction from Step~1 of the proof of Theorem
\ref{thm:main-intro}.  If $S=V_{\eta_0}D_{\nu_0}R_{\sigma_0}$ is the
symplectic matrix used there, then metaplectic covariance reduces the claim to
the frame property of
$$
\calG(\widehat V_{\eta_0}\widehat D_{\nu_0}g^{(0)},
\tfrac pq\Z\times\Z),
\quad g^{(0)}:=\widehat R_{\sigma_0}g.
$$
The function $g^{(0)}$ is again a nontrivial linear combination with the same
largest Hermite index.  Since the argument below applies to every such linear
combination, we relabel $g^{(0)}$ as $g$ for the remainder of the proof.

\textbf{Step 2.} We begin analogously to Step 2 of the proof of Theorem \ref{thm:main-intro} and look at the Zak transform of $\widehat{R}_i\widehat{V}_{\eta_0}\widehat{D}_{\nu_0} g$. We select $\eta\in\R$ and $\nu>0$ such that 
\begin{equation}
    R_i V_{\eta_0} D_{\nu_0} = V_\eta D_\nu R_\sigma.
\end{equation}
By \eqref{eq:lin_comb_sigma_action}, there exists a non-trivial linear combination $\Tilde{g}$ of $h_0,\dots,h_n$ such that
\begin{equation}
    \widehat{V}_\eta \widehat{D}_\nu \widehat{R}_\sigma g = \widehat{V}_\eta \widehat{D}_\nu \Tilde{g}.
\end{equation}
Now, we compute the entries of the matrix 
\begin{equation}
         \QQ(x,\omega)_{rs} = \zak (\widehat{V}_{\eta_0} \widehat{D}_{\nu_0}g)( x+\tfrac{pr}{q}, \omega +\tfrac{s}{p}) ,\quad 0\leq r\leq q-1, \, 0\leq s\leq p-1.
     \end{equation}
     As previously seen, we have
     \begin{equation}
         \begin{split}
        \zak (\widehat{V}_{\eta_0} \widehat{D}_{\nu_0}g)( x, \omega)  &=  e^{2\pi i x\omega }\zak (\widehat{V}_{\eta} \widehat{D}_{\nu}\widehat{R}_\sigma g)(\omega , -x) \\
         & =  e^{2\pi i x\omega }\zak (\widehat{V}_{\eta} \widehat{D}_{\nu}\Tilde{g})(\omega, -x).
         \end{split}
     \end{equation}
\textbf{Step 3.} We now compute $\zak (\widehat{V}_{\eta} \widehat{D}_{\nu}\Tilde{g})$ with respect to the non-orthogonal basis $\Tilde{h}_m(t) = t^m h_0(t)$, $m\in\N_0$. Since $\Tilde{h}_0,\dots, \Tilde{h}_n$ is a basis of the span of $h_0,\dots, h_n$, there exist coefficients $c_d'\in\C$, $c_n'\neq 0$, such that 
\begin{equation}
   \Tilde{g}(t) = \sum\limits_{d=0}^n c_d' t^d h_0(t),\quad c_d'\in\C,\ c_n'\neq 0.
\end{equation}
By \eqref{eq:delta_theta_ab}, we have
\begin{equation}
\begin{split}
          e^{2\pi i x\omega }\zak (\widehat{V}_{\eta} \widehat{D}_{\nu}\Tilde{h}_d)(\omega, -x)
     &= \nu^{-(d+1/2)}e^{2\pi i x\omega }\sum\limits_{k\in\Z} (k+\omega)^d e^{\pi i (\eta+i/\nu^2) (k+\omega)^2} e^{2\pi i k x} \\
     &=\nu^{-(d+1/2)} \sum\limits_{k\in\Z} (k+\omega)^d e^{\pi i (\eta+i/\nu^2) (k+\omega)^2} e^{2\pi i (k+\omega) x}\\
     & = \nu^{-(d+1/2)}\dd_z^d \thetac{\omega}{0}(x,\tau),\quad \tau = \eta+i/\nu^2.
\end{split}
\end{equation}
Let $P(X) = \sum_{d=0}^n \nu^{-(d+1/2)} c_d'X^d$. Then 
\begin{equation}
    e^{2\pi i x\omega }\zak (\widehat{V}_{\eta} \widehat{D}_{\nu}\Tilde{g})(\omega, -x) = \sum\limits_{d=0}^n \nu^{-(d+1/2)} c_d'\dd_z^d\thetac{\omega}{0}(x,\tau) = P(\dd_z)\thetac{\omega}{0}(x,\tau)
\end{equation}
and 
\begin{equation}
    \QQ(x,\omega)_{rs} = P(\dd_z)\thetac{\tfrac{p\omega+ s}{p}}{0}(p(\tfrac{x}{p}+\tfrac{r}{q}),\tau).
\end{equation}
\textbf{Step 4.}
We apply 
Theorem~\ref{thm:primitive-derivative} for the space $\calT_{p,p\omega,0}(\tau/p)$. 
Define
$$
\widetilde P(X):=P(X/p).
$$
We have $\deg \widetilde P=n$. By Proposition \ref{prop:elem_properties_calTp}~(ii)
and the chain rule, 
the matrix $\QQ(x,\omega)$ is the matrix representation of the linear map
\begin{equation}
    \calT_{p,p\omega ,0}(\tfrac{\tau}{p})\to \C^{q},\quad f \mapsto
\left(\widetilde P(\dd_z)f\left(\tfrac{x}{p}+\tfrac{r}{q}\right)\right)_{r=0}^{q-1}.
\end{equation}
Since $(p,q)=1$ and
$q\geq np+2$, the matrix $\QQ(x,\omega)$ has full rank for every $x$ and every
$\omega$.  

With the same quasi-periodicity and Wiener-amalgam arguments, this implies that $\calG(g,A\Z^2)$ is a frame.
\end{proof}
Corollary \ref{cor:higher_order} and the sufficient bound $q\geq np+2$ are optimal in the following sense: 
the authors in \cite{UlanovskiiZlotnikov2026} proved that there exist non-trivial linear combinations $g_n$ of Hermite functions up to order $n$ such that the system $\calG(g_n,\tfrac{1}{n+1}\Z\times\Z)$ is not a Gabor frame (in this case, $p=1$ and $q=n+1$). 
On the other hand, if $g=h_n$ is a Hermite function of order $n$ and $\Lambda$ is a rectangular lattice $\Lambda = a\Z\times b\Z$, then  $D(\Lambda)=\tfrac{q}{p}$ with $q\geq np+2$ necessarily satisfies $D(\Lambda)>n$. 
Together with Faulhuber and Zlotnikov, the second author showed
in \cite{FaulhuberShafkulovskaZlotnikov_part2}
that for sufficiently large $n$, a density $D(\Lambda)\in o(n)$ is already sufficient for  $\calG(h_n,a\Z\times b\Z)$ to be a frame, so Corollary \ref{cor:higher_order} is not optimal in this case.

\section{Appendix: Lean formalization}\label{sec:lean}

We now describe the formalization of Theorem~\ref{thm:main-intro} in Lean 4 \cite{Moura2021Lean4}, which builds on Mathlib \cite{mathlib2020}, the mathematical library of Lean.  We remark that there are now several recent papers in analysis with complete Lean verifications, including  \cite{armstrong2026formalization, pont2026cantor,hariharan2026milestone,ilin2026semi, liehr2026grothendieck,miller2026formalization}. 

Theorem~\ref{thm:main-intro} is an equivalence whose substance is the implication $(2)\implies(1)$. The Lean formalization verifies this implication, namely the following statement.

\begin{theorem*}
Let $\Lambda=M\Z^2$, $M\in\mathrm{GL}(2,\R)$, be a lattice whose density satisfies $D(\Lambda)=\frac qp$, where $p,q\in\N$ are coprime and $q\geq p+2$. Then $\calG(h_1,\Lambda)$ is a frame for $L^2(\R)$.
\end{theorem*}

The converse implication $(1)\implies(2)$ is not part of the formalization, since, as explained after Theorem~\ref{thm:main-intro}, it follows from results that are well established in the literature. Consequently, the formalization covers precisely the direction of Theorem~\ref{thm:main-intro} that is new. The purpose of this appendix is to make the correspondence between the Lean code and the mathematical statement precise and transparent. To describe this correspondence, a small amount of Lean syntax is needed: in Lean, a definition is introduced with \lean{def}, or with \lean{abbrev} in the case of a mere abbreviation, while a proved statement is introduced with \lean{lemma} or \lean{theorem}. In a statement, the hypotheses are listed after the name of the statement, and each hypothesis is given a name by which it can be referred to. The expression after the last colon is the conclusion.

\subsection{The space $L^2(\R)$ and the first Hermite function}

We begin by defining the Hilbert space $L^2(\R)$. In Mathlib, \lean{volume} denotes the Lebesgue measure on $\R$, and \lean{Lp ℂ 2 volume} is Lean's notation for the Hilbert space $L^2(\R)=L^2(\R,\C)$. We abbreviate this as follows.

\begin{leancode}
abbrev L2 := Lp ℂ 2 (volume : Measure ℝ)
\end{leancode}

The first Hermite function is defined directly by the formula $h_1(t)=te^{-\pi t^2}$.

\begin{leancode}
def h1 : ℝ → ℂ :=
    fun t => (t : ℂ) * Complex.exp (-(Real.pi : ℂ) * (t : ℂ) ^ 2)
\end{leancode}

Here, \lean{ℝ → ℂ} indicates that \lean{h1} is a function from $\R$ to $\C$, \lean{Complex.exp} is the complex exponential function, and \lean{Real.pi} is the number $\pi$. The annotations \lean{(t : ℂ)} and \lean{(Real.pi : ℂ)} tell Lean to regard the real numbers $t$ and $\pi$ as complex numbers. As pointed out in the introduction, the normalization of $h_1$ does not affect the frame property, and \lean{h1} is precisely the function $h_1$ used throughout the paper.

\subsection{Lattices and Gabor atoms}

A lattice $\Lambda=M\Z^2$ is determined by its generating matrix $M \in \R^{2\times 2}$. In Lean, the expression \lean{M : Matrix (Fin 2) (Fin 2) ℝ} corresponds to a real $2\times2$ matrix. The type \lean{Fin 2} consists of the indices $0$ and $1$, so that \lean{M i j} denotes the entry of $M$ in row $i$ and column $j$. For $(m,n)\in\Z^2$, the time-frequency shift $M_\omega T_x h_1$ of $h_1$ along the lattice point $(x,\omega)^{T}=M(m,n)^{T}$ is first defined as an ordinary complex-valued function on $\R$.

\begin{leancode}
def h1Atom (M : Matrix (Fin 2) (Fin 2) ℝ) (m n : ℤ) : ℝ → ℂ := fun t =>
    let x := M 0 0 * (m : ℝ) + M 0 1 * (n : ℝ)
    let ω := M 1 0 * (m : ℝ) + M 1 1 * (n : ℝ)
    Complex.exp (2 * (Real.pi : ℂ) * Complex.I * (ω : ℂ) * (t : ℂ)) * h1 (t - x)
\end{leancode}

The keyword \lean{let} introduces the local abbreviations \lean{x} and \lean{ω} for the two coordinates of the lattice point $M(m,n)^{T}$, where the integers \lean{m n : ℤ} are regarded as real numbers via \lean{(m : ℝ)} and \lean{(n : ℝ)}. The last line is the function $t\mapsto e^{2\pi i\omega t}h_1(t-x)$. As $(m,n)$ ranges over $\Z^2$, the point $M(m,n)^{T}$ ranges over $\Lambda$, so the functions \lean{h1Atom M m n} are precisely the elements of the Gabor system $\calG(h_1,\Lambda)$.

To use this function in the Hilbert space $L^2(\R)$, one must first verify that it is square-integrable. This is the content of the following lemma.

\begin{leancode}
lemma h1Atom_memL2 (M : Matrix (Fin 2) (Fin 2) ℝ) (m n : ℤ) :
    MemLp (h1Atom M m n) 2 (volume : Measure ℝ)
\end{leancode}

The predicate \lean{MemLp} expresses membership of an ordinary function in an $L^p$-space, so the lemma \lean{h1Atom_memL2} states that every time-frequency shift of the first Hermite function is measurable and square-integrable with respect to the Lebesgue measure. Mathematically, this is immediate from the Gaussian decay of $h_1$, but Lean requires an explicit proof. The lemma is then used to construct the corresponding element of $L^2(\R)$.

\begin{leancode}
def h1AtomL2 (M : Matrix (Fin 2) (Fin 2) ℝ) (m n : ℤ) : L2 :=
    (h1Atom_memL2 M m n).toLp
\end{leancode}

The operation \lean{toLp} converts the function \lean{h1Atom M m n}, together with the proof of its square-integrability, into an element of \lean{L2}.

\subsection{Gabor frames}

The frame property of the Gabor system $\calG(h_1,M\Z^2)$ is formalized by the two-sided frame inequality \eqref{eq:def:frame}.

\begin{leancode}
def IsGaborFrame (M : Matrix (Fin 2) (Fin 2) ℝ) : Prop :=
  ∃ A B : ℝ, 0 < A ∧ A ≤ B ∧ ∀ f : L2,
    A * ‖f‖ ^ 2 ≤ ∑' (m : ℤ) (n : ℤ), ‖⟪f, h1AtomL2 M m n⟫_ℂ‖ ^ 2 ∧
    ∑' (m : ℤ) (n : ℤ), ‖⟪f, h1AtomL2 M m n⟫_ℂ‖ ^ 2 ≤ B * ‖f‖ ^ 2
\end{leancode}

This definition states the existence of frame bounds $A,B\in\R$ with $0<A\leq B$ such that the frame inequalities hold for every \lean{f : L2}, i.e., for every $f\in L^2(\R)$. Here \lean{∑'} denotes Mathlib's infinite sum, \lean{⟪f, h⟫_ℂ} denotes the complex inner product of $f$ and $h$, and \lean{‖f‖} denotes the norm of \lean{f}. Since $(m,n)\mapsto M(m,n)^{T}$ enumerates the lattice $M\Z^2$, the double sum over \lean{m n : ℤ} is the sum over $(x,\omega)\in\Lambda$ in \eqref{eq:def:frame}. Consequently, \lean{IsGaborFrame M} expresses precisely that
\begin{equation*}
    A\norm{f}^2
    \leq
    \sum_{(x,\omega)\in M\Z^2}|\langle f,M_\omega T_x h_1\rangle|^2
    \leq
    B\norm{f}^2,\quad f\in L^2(\R).
\end{equation*}

\subsection{Formalization of the main result}

With the above definitions in place, the implication $(2)\Rightarrow(1)$ of Theorem~\ref{thm:main-intro} is stated in Lean as follows.

\begin{leancode}
theorem MainResult
    (M : Matrix (Fin 2) (Fin 2) ℝ) {p q : ℕ}
    (hyp1 : M.det ≠ 0)
    (hyp2 : Nat.Coprime p q)
    (hyp3 : q ≥ p + 2)
    (hyp4 : |M.det|⁻¹ = (q : ℝ) / (p : ℝ)) :
    IsGaborFrame M
\end{leancode}

The above theorem takes as input a real $2\times2$ matrix \lean{M}, which generates the lattice $\Lambda=M\Z^2$, and two natural numbers \lean{p} and \lean{q}. The notation \lean{{p q : ℕ}} means that Lean treats $p,q$ as implicit arguments, whose values are inferred from the hypotheses whenever the theorem is applied. The hypothesis \lean{hyp1} states that $\det M\neq0$, so that $M\in\mathrm{GL}(2,\R)$ and $M\Z^2$ is indeed a lattice. The hypotheses \lean{hyp2} and \lean{hyp3} state that $p$ and $q$ are coprime and that $q\geq p+2$, respectively. Since the density of $M\Z^2$ is the reciprocal of the area $|\det M|$ of a fundamental domain, the hypothesis \lean{hyp4}, in which \lean{(q : ℝ)} and \lean{(p : ℝ)} denote $q$ and $p$ regarded as real numbers, states that $D(\Lambda)=\frac qp$. Finally, the expression after the last colon is the conclusion, namely that $\calG(h_1,\Lambda)$ is a frame for $L^2(\R)$.

We note that Lean's natural numbers include $0$, whereas Theorem~\ref{thm:main-intro} concerns $p,q\in\N=\{1,2,\ldots\}$. However, the case $p=0$ is excluded by the hypotheses: Lean follows the convention that division by zero equals zero, so for $p=0$ the hypothesis \lean{hyp4} would read $|\det M|^{-1}=0$, which contradicts \lean{hyp1}. Together with $q\geq p+2$, the hypotheses of \lean{MainResult} thus describe exactly the lattices in Theorem~\ref{thm:main-intro}(2). Accordingly, \lean{MainResult} formalizes the implication $(2)\Rightarrow(1)$ of Theorem~\ref{thm:main-intro}, that is, the theorem stated at the beginning of this appendix.

\subsection{Source code}\label{sec:lean-source}

The complete source code of the Lean formalization is available at
\begin{center}
\url{https://github.com/lukasliehr/h1_frames}.
\end{center}
The file \lean{Showcase.lean} contains the definitions and statements described above, with the two proofs, of \lean{h1Atom_memL2} and of \lean{MainResult}, replaced by the placeholder \lean{sorry}. This file depends on nothing but Mathlib, and verifying that the formalization faithfully captures the theorem stated at the beginning of this appendix therefore only requires reading this file. The companion file \lean{Showcase_WithProofs.lean} contains the same definitions and statements, but each placeholder is replaced by a fully verified proof, whose correctness is certified by the Lean kernel during compilation.

\section{Acknowledgements}
I.~S.~was funded in part or in whole by the Austrian Science Fund (FWF)  [\href{https://doi.org/10.55776/Y1199}{10.55776/Y1199}].
L.~L.~is grateful to the Azrieli Foundation for the award of an Azrieli Fellowship and acknowledges the support of this research by ISF Grant No.~854/25.
For open access purposes, the authors have applied a CC BY public copyright license to any author accepted manuscript version arising from this submission.

\bibliographystyle{abbrv}
\bibliography{bibfile}

\begin{thebibliography}{10}

\bibitem{armstrong2026formalization}
S.~Armstrong and J.~Kempe.
\newblock Formalization of {D}e {G}iorgi--{N}ash--{M}oser {T}heory in {L}ean.
\newblock {\em arXiv preprint arXiv:2604.05984}, 2026.

\bibitem{Balian1981}
R.~Balian.
\newblock Un principe d'incertitude fort en th\'eorie du signal ou en
  m\'ecanique quantique.
\newblock {\em C. R. Acad. Sci. Paris S\'er. II M\'ec. Phys. Chim. Sci. Univers
  Sci. Terre}, 292(20):1357--1362, 1981.

\bibitem{BauerSzemberg1997}
T.~Bauer and T.~Szemberg.
\newblock Higher order embeddings of abelian varieties.
\newblock {\em Math. Z.}, 224(3):449--455, 1997.

\bibitem{pont2026cantor}
J.~de~Dios~Pont, L.~Liehr, and M.~A. Taylor.
\newblock Cantor measures with odd base do not admit {F}ourier frames.
\newblock {\em arXiv preprint arXiv:2607.08656}, 2026.

\bibitem{Moura2021Lean4}
L.~de~Moura and S.~Ullrich.
\newblock The {Lean} 4 theorem prover and programming language.
\newblock In {\em Automated Deduction -- CADE 28}, volume 12699 of {\em Lecture
  Notes in Computer Science}, pages 625--635, Cham, 2021. Springer.

\bibitem{FarkasKra2001}
H.~M. Farkas and I.~Kra.
\newblock {\em Theta constants, {R}iemann surfaces and the modular group},
  volume~37 of {\em Graduate Studies in Mathematics}.
\newblock American Mathematical Society, Providence, RI, 2001.
\newblock An introduction with applications to uniformization theorems,
  partition identities and combinatorial number theory.

\bibitem{Faulhuber2020}
M.~Faulhuber.
\newblock On the parity under metapletic operators and an extension of a result
  of {L}yubarskii and {N}es.
\newblock {\em Results Math.}, 75(1):Paper No. 8, 16, 2020.

\bibitem{Faulhuber2026wirtinger}
M.~Faulhuber.
\newblock On a fundamental barrier of the {W}irtinger criterion for {G}abor
  systems with odd functions.
\newblock {\em Results Appl. Math.}, 30:Paper No. 100719, 2026.

\bibitem{FaulhuberEtAl2025}
M.~Faulhuber, I.~Shafkulovska, and I.~Zlotnikov.
\newblock On the frame property of {H}ermite functions and exploration of their
  frame sets.
\newblock {\em J. Fourier Anal. Appl.}, 31(2):Paper No. 21, 30, 2025.

\bibitem{FaulhuberShafkulovskaZlotnikov_part2}
M.~Faulhuber, I.~Shafkulovska, and I.~Zlotnikov.
\newblock Asymptotic safety regions for {G}abor frames generated by {H}ermite
  functions.
\newblock {\em In preparation}, 2026.

\bibitem{Folland1989}
G.~B. Folland.
\newblock {\em Harmonic analysis in phase space}, volume 122 of {\em Annals of
  Mathematics Studies}.
\newblock Princeton University Press, Princeton, NJ, 1989.

\bibitem{GhoshSelvan2025}
R.~Ghosh and A.~Antony~Selvan.
\newblock On {G}abor frames generated by {B}-splines, totally positive
  functions, and {H}ermite functions.
\newblock {\em Appl. Numer. Math.}, 207:1--23, 2025.

\bibitem{Groechenig2001}
K.~Gr\"{o}chenig.
\newblock {\em Foundations of time-frequency analysis}.
\newblock Applied and Numerical Harmonic Analysis. Birkh\"{a}user Boston, Inc.,
  Boston, MA, 2001.

\bibitem{Groechenig2014}
K.~Gr\"ochenig.
\newblock The mystery of {G}abor frames.
\newblock {\em J. Fourier Anal. Appl.}, 20(4):865--895, 2014.

\bibitem{GroechenigKoppensteiner2019}
K.~Gr\"ochenig and S.~Koppensteiner.
\newblock Gabor frames: characterizations and coarse structure.
\newblock In {\em New trends in applied harmonic analysis. {V}ol. 2---harmonic
  analysis, geometric measure theory, and applications}, Appl. Numer. Harmon.
  Anal., pages 93--120. Birkh\"auser/Springer, Cham, 2019.

\bibitem{GroechenigLyubarskii2007}
K.~Gr\"ochenig and Y.~Lyubarskii.
\newblock Gabor frames with {H}ermite functions.
\newblock {\em C. R. Math. Acad. Sci. Paris}, 344(3):157--162, 2007.

\bibitem{GroechenigLyubarskii2009}
K.~Gr\"ochenig and Y.~Lyubarskii.
\newblock Gabor (super)frames with {H}ermite functions.
\newblock {\em Math. Ann.}, 345(2):267--286, 2009.

\bibitem{hariharan2026milestone}
S.~Hariharan, C.~Birkbeck, S.~Lee, H.~K.~G. Ma, B.~Mehta, A.~Poiroux, and
  M.~Viazovska.
\newblock A {M}ilestone in {F}ormalization: {T}he {S}phere {P}acking {P}roblem
  in {D}imension 8.
\newblock {\em arXiv preprint arXiv:2604.23468}, 2026.

\bibitem{Heil2007}
C.~Heil.
\newblock History and evolution of the density theorem for {G}abor frames.
\newblock {\em J. Fourier Anal. Appl.}, 13(2):113--166, 2007.

\bibitem{HorstEtAl2025}
A.~Horst, J.~Lemvig, and A.~E. Videb\ae~k.
\newblock On the non-frame property of {G}abor systems with {H}ermite
  generators and the frame set conjecture.
\newblock {\em Appl. Comput. Harmon. Anal.}, 76:Paper No. 101747, 17, 2025.

\bibitem{ilin2026semi}
V.~Ilin.
\newblock Semi-autonomous formalization of the {V}lasov-{M}axwell-{L}andau
  equilibrium.
\newblock {\em arXiv preprint arXiv:2603.15929}, 2026.

\bibitem{Knapp1996}
A.~W. Knapp.
\newblock {\em Lie groups beyond an introduction}, volume 140 of {\em Progress
  in Mathematics}.
\newblock Birkh\"auser Boston, Inc., Boston, MA, 1996.

\bibitem{Krazer1903}
A.~Krazer.
\newblock {\em Lehrbuch der Thetafunktionen}, volume~12.
\newblock BG Teubner, 1903.

\bibitem{AKIZ}
A.~Kulikov and I.~Zlotnikov.
\newblock Gabor frames for the first {H}ermite function with every admissible
  density.
\newblock {\em In preparation}, 2026.

\bibitem{Lemvig2017}
J.~Lemvig.
\newblock On some {H}ermite series identities and their applications to {G}abor
  analysis.
\newblock {\em Monatsh. Math.}, 182(4):899--912, 2017.

\bibitem{liehr2026grothendieck}
L.~Liehr, M.~A. Taylor, and P.~Yu.
\newblock Grothendieck's theorem for {B}essel sequences.
\newblock {\em arXiv preprint arXiv:2608.12280}, 2026.

\bibitem{Low1985}
F.~Low.
\newblock Complete sets of wave packets.
\newblock In C.~DeTar, J.~Finkelstein, and C.-I. Tan, editors, {\em A Passion
  for Physics: Essays in Honor of Geoffrey Chew}, pages 17--22. World
  Scientific, Singapore, 1985.

\bibitem{lyubarskii2013gabor}
Y.~Lyubarskii and P.~G. Nes.
\newblock Gabor frames with rational density.
\newblock {\em Applied and Computational Harmonic Analysis}, 34(3):488--494,
  2013.

\bibitem{Lyubarskii1992}
Y.~I. Lyubarski\u{\i}.
\newblock Frames in the {B}argmann space of entire functions.
\newblock In {\em Entire and subharmonic functions}, volume~11 of {\em Adv.
  Soviet Math.}, pages 167--180. Amer. Math. Soc., Providence, RI, 1992.

\bibitem{miller2026formalization}
J.~K. Miller.
\newblock A {F}ormalization of the {M}ean-{F}ield {D}erivation of the {V}lasov
  {E}quation: {AI-Assisted Lean Formalization as a Strategy Game}.
\newblock {\em arXiv preprint arXiv:2607.08986}, 2026.

\bibitem{MumfordTataI}
D.~Mumford.
\newblock {\em Tata lectures on theta. {I}}.
\newblock Modern Birkh\"auser Classics. Birkh\"auser Boston, Inc., Boston, MA,
  2007.
\newblock With the collaboration of C. Musili, M. Nori, E. Previato and M.
  Stillman, Reprint of the 1983 edition.

\bibitem{RosengrenSchlosser2006}
H.~Rosengren and M.~Schlosser.
\newblock Elliptic determinant evaluations and the {M}acdonald identities for
  affine root systems.
\newblock {\em Compos. Math.}, 142(4):937--961, 2006.

\bibitem{Seip1992}
K.~Seip.
\newblock Density theorems for sampling and interpolation in the
  {B}argmann-{F}ock space. {I}.
\newblock {\em J. Reine Angew. Math.}, 429:91--106, 1992.

\bibitem{SeipWallsten1992}
K.~Seip and R.~Wallst\'en.
\newblock Density theorems for sampling and interpolation in the
  {B}argmann-{F}ock space. {II}.
\newblock {\em J. Reine Angew. Math.}, 429:107--113, 1992.

\bibitem{mathlib2020}
{The mathlib Community}.
\newblock The {Lean} mathematical library.
\newblock In {\em Proceedings of the 9th ACM SIGPLAN International Conference
  on Certified Programs and Proofs (CPP 2020)}, pages 367--381, New York, NY,
  2020. ACM.

\bibitem{UlanovskiiZlotnikov2026}
A.~Ulanovskii and I.~Zlotnikov.
\newblock Periodic non-uniqueness sets for shift-invariant spaces and
  parity-based obstructions to the frame property for {G}abor systems, 2026.
\newblock Preprint, arXiv:2606.31450.

\bibitem{ZibulskiZeevi1993}
M.~Zibulski and Y.~Zeevi.
\newblock Oversampling in the {G}abor scheme.
\newblock {\em Trans. Sig. Proc.}, 41(8):2679–2687, Aug. 1993.

\bibitem{ZibulskiZeevi1997}
M.~Zibulski and Y.~Y. Zeevi.
\newblock Analysis of multiwindow {G}abor-type schemes by frame methods.
\newblock {\em Appl. Comput. Harmon. Anal.}, 4(2):188--221, 1997.

\end{thebibliography}

\end{document}